\documentclass[3p,times,reqno]{elsarticle-1}

\usepackage{ecrc-1}

\usepackage{enumitem}

\firstpage{1}

\runauth{U.N. Katugampola}

\usepackage{color,soul} 
\usepackage{lmodern}
\usepackage{amsmath}

\usepackage{amssymb,amsthm}
\usepackage{verbatim}

 \biboptions{sort&compress}
\newtheorem{theorem}{Theorem}[section]

\newtheorem{corollary}[theorem]{Corollary}
\theoremstyle{definition}
\newtheorem{definition}[theorem]{Definition}

\theoremstyle{remark}
\newtheorem{remark}[theorem]{Remark}

\begin{document}
\begin{frontmatter}
\title{New fractional integral unifying six existing fractional integrals}

%
\author{Udita N. Katugampola\corref{corresp}}

%

\cortext[corresp]{Corresponding author.  Tel.: +13028312694. \\
\indent {\it E-mail address:} uditanalin@yahoo.com (U. N. Katugampola)  }
%
\address{Department of Mathematical Sciences, University of Delaware, Newark, DE 19716, U.S.A.}

\begin{abstract}
In this paper we introduce a new fractional integral that generalizes six existing fractional integrals, namely, Riemann-Liouville, Hadamard, Erd\'elyi-Kober, Katugampola, Weyl and Liouville fractional integrals in to one form. Such a generalization takes the form 
\[
    \left({}^{\rho}\mathcal{I}^{\alpha, \beta}_{a+;\eta, \kappa}f\right)(x)=\frac{\rho^{1-\beta}x^{\kappa}}{\Gamma(\alpha)}\int_a^x \frac{\tau^{\rho \eta +\rho-1}}{(x^\rho-\tau^\rho)^{1-\alpha}}f(\tau)\text{d}\tau, \quad 0\leq a < x < b \leq \infty.
\]
A similar generalization is not possible with the Erd\'elyi-Kober operator though there is a close resemblance with the operator in question. We also give semigroup, boundedness, shift and integration-by-parts formulas for completeness. 
\end{abstract}
\begin{keyword}
Riemann-Liouville integral \sep Hadamard integral \sep Erd\'elyi-Kober integral \sep Katugampola integral 
 
%
%
%
\end{keyword}
\end{frontmatter}

%

%
%
%

\setcounter{section}{1}
\section*{}
Until very recently, the fractional calculus had been a purely mathematical subject without apparent applications. Nowadays, it plays a major role in modeling anomalous behavior and memory effects and  
appears naturally in modeling long-term behaviors, especially in the areas of viscoelastic materials and viscous fluid dynamics  \cite{what1,quo}. Fractional integrals alone, without its counterpart, naturally appear in certain modeling and theoretical problems, for example, probability theory \cite{pro1}, surface-volume reaction problems \cite{survol}, anomalous diffusion \cite{pagnini}, porous medium equations \cite{ploc1,ploc2}, and numerical analysis \cite{almeida}, among other applications.   
Now, consider the following generalized integral.

\begin{definition}\label{newI} Let $f\in\textit{X}^p_c(a,b)$ \cite{udita2}, $\alpha>0, \rho,\eta,\kappa \in \mathbb{R}$. The left-sided generalized fractional integral is defined by, 
\begin{equation}
\left({}^\rho I^{\alpha, \beta}_{a+; \eta, \kappa}f\right)(x) = \frac{\rho^{1-\beta}x^{\kappa}}{\Gamma({\alpha})} \int^x_a \frac{\tau^{\rho(\eta+1)-1} }{(x^{\rho } - \tau^{\rho })^{1-\alpha}}f(\tau)\,d\tau, \quad (0 \leq a < x <b \leq \infty),
\label{lint}
\end{equation}
if the integral exists.
\end{definition}
It can be seen that this integral generalizes four existing integrals. For $\eta=0, \kappa=0$, the \textit{Riemann-Liouville} fractional integral is obtained when $\rho = 1$, while the \textit{Katugampola} integral is obtained if $\beta=\alpha$, further, in this case, when $\rho \rightarrow 0^+$, the integral coincides with \textit{Hadamard} integral, which can be easily verified using L'Hospital rule. Now, for $\beta=0, \kappa=-\rho(\alpha+\eta)$, and any $\eta$, it gives the \textit{Erd\'{e}lyi-Kober}(type) operator. It should be remarked that $\rho^{1-\beta}$ is complex when $\beta\notin \mathbb{Z}$ and $\rho <0$ and can be treated using theory of complex analysis considering appropriate branches.  

Fractional integrals sometimes work in pairs, specially in variational calculus \cite{almeida}. The corresponding right-sided fractional integral can be defined as,
\begin{equation}
\left({}^\rho I^{\alpha, \beta}_{b-; \eta, \kappa}f\right)(x) = \frac{\rho^{1-\beta}x^{\rho\eta}}{\Gamma({\alpha})} \int^b_x \frac{\tau^{\kappa+\rho-1} }{(\tau^{\rho } - x^{\rho })^{1-\alpha}}f(\tau)\,d\tau, \quad (0 \leq a < x <b \leq \infty),
\label{rint}
\end{equation}
if the integral exists. 

In the Definitions~(\ref{lint}) and (\ref{rint}), we can also consider the cases $a=-\infty$ and $b=\infty$, respectively, and are known in the literature as Weyl and Liouville type integrals, respectively. Such integrals are corresponding to \textit{infinite memory} effects and have applications in financial mathematics and diffusion models \cite{Tarasov,Baillie}. 

\begin{figure}[ht]
	\centering
		\includegraphics{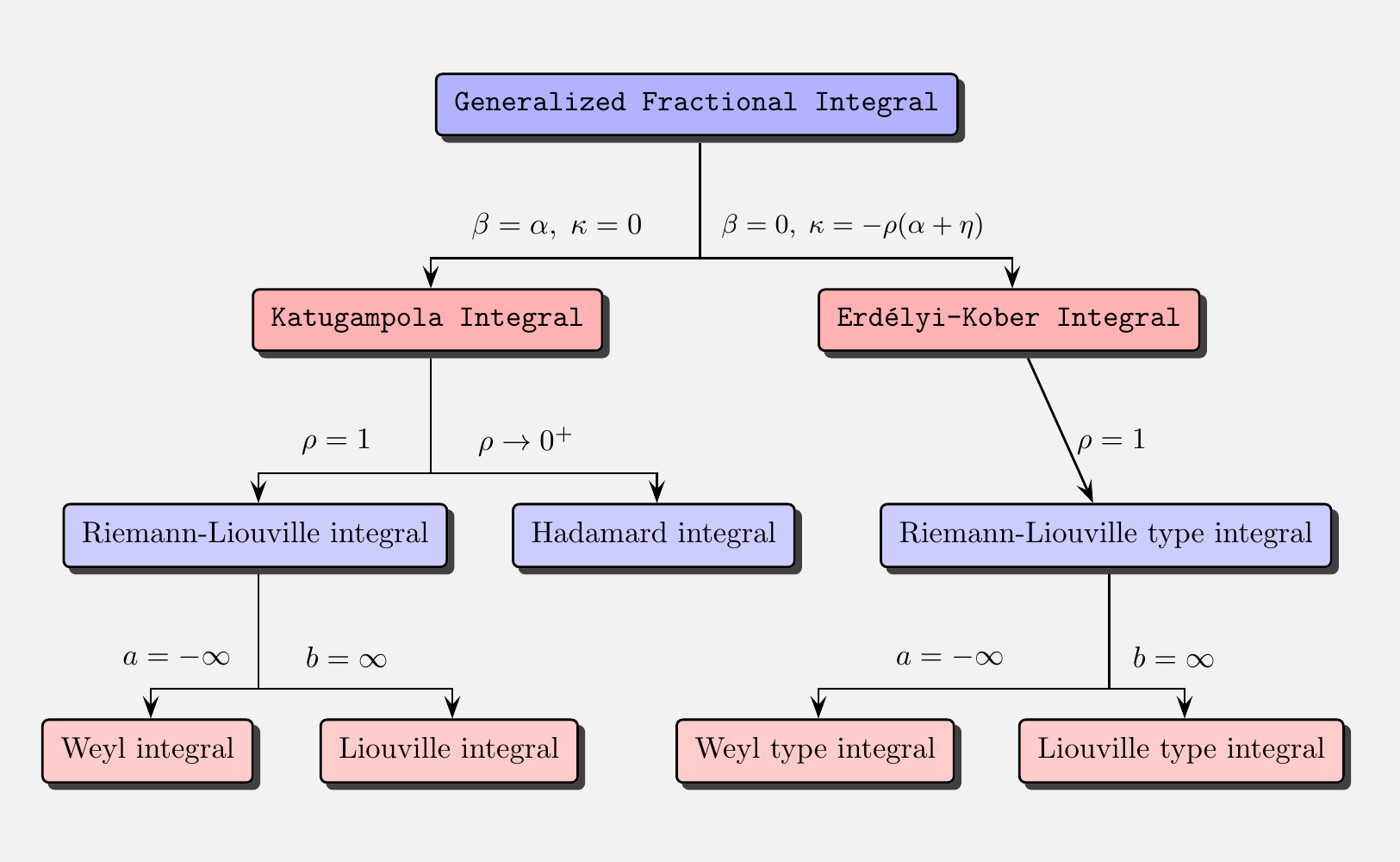}
	\label{fig:Fig1}
	\caption{Generalized fractional integral}
\end{figure}


Let us note that, using the change of variable $u=(\tau/x)^\rho$, the integral (\ref{lint}) can be rewritten in the Riemann-Liouville form as 
\begin{equation}
\left({}^\rho I^{\alpha, \beta}_{a+; \eta, \kappa}f\right)(x) = \frac{x^{\kappa+\rho(\alpha+\eta)}}{\rho^\beta\Gamma(\alpha)}\int_0^1(1-u)^{\alpha-1}u^\eta f(xu^{1/\rho}) du, \quad \rho \ne 0.
\label{lint2}
\end{equation}
Further the Riemann-Liouville fractional integral is used to define both the Riemann-Liouville and the Caputo fractional derivatives \cite{key-8, key-9}. To justify the claim about the Hadamard integral, when $\kappa = 0$ and $\beta=\alpha>0$, using L'Hospital rule and taking $\rho \rightarrow 0^+$, we have
\begin{align}
\lim_{\rho \rightarrow 0^+} \frac{\rho^{1-\alpha}}{\Gamma({\alpha})} \int^x_a \frac{\tau^{\rho-1}}{(x^{\rho} - \tau^{\rho})^{1-\alpha}} f(\tau) d\tau &=\frac{1}{\Gamma({\alpha})}  \int^x_a \lim_{\rho \rightarrow 0^+} \bigg(\frac{x^{\rho} - \tau^{\rho}}{\rho}\bigg)^{\alpha -1} \tau^{\rho-1} f(\tau) d\tau \nonumber \\
&= \frac{1}{\Gamma(\alpha)}\int_a^x \Big(\log \frac{x}{\tau} \Big)^{\alpha -1} \frac{f(\tau)}{\tau} d\tau, \nonumber
\end{align}
which is the Hadamard fractional integral \cite[p.110]{key-8}. \textit{It should be pointed out that a similar result is not possible with the Erd\'elyi-Kober operator though there is a close resemblance with the operator in }(\ref{lint}). Recent results about the Hadamard and Hadamard-type integrals such as Hadamard-type fractional calculus \cite{key-3}, composition and semigroup properties \cite{key-1}, Mellin transforms \cite{key-5}, integration by parts formulae \cite{key-6}, \emph{G-transform} representations \cite{key-7}, impulsive differential equations with Hadamard fractional derivative \cite{imp1,boun1}, and Hadamard type fractional differential systems \cite{hsys1} can be found in the literature, among others.

The Katugampola fractional integral was first introduced in \cite{udita1,u-1} as a generalization of $n-$fold integral, and then a simpler version was discussed in \cite{udita2} along with the corresponding fractional derivatives. The Mellin transforms of it were given in \cite{udita3}. The same reference also discusses a new class of generalized Stirling numbers of $2^{nd}$ kind and a recurrence formula for such sequences. Further applications of Katugampola fractional integrals or derivatives are in probability theory \cite{pro1}, variational calculus \cite{almeida}, inequalities \cite{chen1}, Langevin equations \cite{lang1}, Fourier and Laplace transforms \cite{chen2} fractional differential equations \cite{caka} and Numerical analysis \cite{num1}, among others.

\section{Main Results}
For simplicity, we give the following results without proofs. The proofs of the similar results for the Erd\'elyi-Kober type operators can be found in the classical books by Kiryakova \cite{Kirk}, Yakubovich and Luchko \cite[p.54]{Yaku}, McBride \cite[p.123]{McB}, 
Kilbas et al. \cite[Section 2.6]{key-8} and Samko et al. \cite[Section 18.1]{key-9}, and can be generalized to the present case.

For ``sufficiently good'' functions $f, g$ we have the following results.
\begin{enumerate}
	\item [a)] Shift formulae 
	      \begin{equation} \label{shift1}
				\begin{aligned}
				    {}^\rho I^{\alpha, \beta}_{a+; \eta, \kappa}x^{\rho\gamma}f(x) &= \left({}^\rho I^{\alpha, \beta}_{a+; \eta+\gamma, \kappa}f\right)(x), \\
						{}^\rho I^{\alpha, \beta}_{b-; \eta, \kappa}x^{\gamma}f(x) &= \left({}^\rho I^{\alpha, \beta}_{b-; \eta, \kappa+\gamma}f\right)(x).
			  	\end{aligned}
	       \end{equation}
	\item [b)] Composition (index) formulae
	    \begin{equation} \label{comp1}
				\begin{aligned}
				    {}^\rho I^{\alpha_1, \beta_1}_{a+; \eta_1, \kappa_1}{}^\rho I^{\alpha_2, \beta_2}_{a+; \eta_2, -\rho\eta_1}f&= {}^\rho I^{\alpha_1+\alpha_2, \beta_1+\beta_2}_{a+; \eta_2, \kappa_1}f, \\ {}^\rho I^{\alpha_1, \beta_1}_{b-; \eta_1, -\rho\eta_2}{}^\rho I^{\alpha_2, \beta_2}_{b-; \eta_2, \kappa_2}f&= {}^\rho I^{\alpha_1+\alpha_2, \beta_1+\beta_2}_{b-; \eta_1, \kappa_2}f.
			  	\end{aligned}
	       \end{equation}
	\noindent which hold in the corresponding spaces of the functions $f$ if $\alpha_2>0, \alpha_1+\alpha_2 \geq 0$ or $\alpha_2 <0, \alpha_1 >0$ or $\alpha_1<0, \alpha_1+\alpha_2 \leq 0$. (see Theorem 2.5 of \cite{key-9}). 				
	\item [c)] Fractional product-integration formulae 
	    \begin{equation}\label{ipart}
			    \int_a^b x^{\rho-1}f(x)\left({}^\rho I^{\alpha, \beta}_{a+; \eta, \kappa}g\right)(x)dx=\int_a^b x^{\rho-1}g(x)\left({}^\rho I^{\alpha, \beta}_{b-; \eta, \kappa}f\right)(x)dx
			\end{equation}
				Similar results are also valid when $a=0$ and $b=\infty$ and in particular for $\rho=2$ and $\rho=1$.
\end{enumerate}
\noindent The proof of part (a) is straightforward and for completeness we shall prove parts (b) and (c) later in the paper. Here we mean by a ``sufficiently good'' is that $f\in\textit{X}^p_c(a,b)$ \cite{udita2}, that is $t^{c-1/p}f(t) \in \textit{L}_p(a,b)$. We need such additional conditions to guarantee convergence. Further results on such conditions can be found, for example, in \cite{Kirk} and \cite{key-9}. In such a space we have the following boundedness result.


\begin{theorem} \label{eq:th1}
Let $\alpha > 0,\, 1 \leq p \leq \infty,\, 0 <a < b < \infty$ and let $\rho \in \mathbb{R}$ and $c \in \mathbb{R}$ be such that $\rho \geq c$. Then the operator ${}^\rho I^{\alpha, \beta}_{a+; \eta, \kappa}$ is bounded in $\textit{X}^p_c(a,b)$ and 
\begin{equation}
\left\|{}^\rho I^{\alpha, \beta}_{a+; \eta, \kappa}f \right\|_{\textit{X}^p_c} \leq K\left\|f\right\|_{\textit{X}^p_c}
\end{equation}
\label{eq:thm1}
where 
\begin{equation}
K = \frac{\rho^{1-\beta}b^{\rho(\alpha+\eta)+\kappa}}{\Gamma(\alpha)}\int^{\frac{b}{a}}_1 \frac{u^{c-\rho(\alpha+\eta)-1}}{(u^{\rho}-1)^{1-\alpha}}du,\quad \quad \rho \ne 0, \kappa \in \mathbb{R}, \eta \geq 0.  \label{eq:con1}
\end{equation}


\end{theorem}

\begin{proof} The proof is 
similar to the case of Katugampola integral \cite[Theorem 3.1]{udita1}. 
First consider the case $1 \leq p \leq \infty$. Since $f \in \textit{X}^p_c(a,b)$, then $t^{c-1/p}f(t) \in \textit{L}_p(a,b)$ and we can apply the generalized Minkowsky inequality. We thus have
\begin{align}
\|{}^\rho I^{\alpha, \beta}_{a+; \eta, \kappa}f\|_{\textit{X}^p_c}
&=\left(\int_a^{b}x^{cp}\left|\frac{\rho^{1-\beta}x^\kappa}{\Gamma(\alpha)}\int^x_a\big(x^\rho-t^\rho\big)^{\alpha-1}t^{\rho\eta+\rho-1} f(t) dt\right|^p\frac{dx}{x} \right)^\frac{1}{p} \nonumber \\
&=\frac{\rho^{\alpha-\beta}}{\Gamma(\alpha)}\left(\int_a^{b}\left|\int^x_a x^{c+\kappa-\frac{1}{p}}\, t^{\rho\eta+\rho-1} \left(\frac{x^\rho-t^\rho}{\rho}\right)^{\alpha-1} f(t) dt\right|^p dx \right)^\frac{1}{p}\nonumber \\
&=\frac{\rho^{\alpha-\beta}}{\Gamma(\alpha)}\left(\int_a^{b}\left|\int^x_a x^{c+\kappa-\frac{1}{p}}\, t^{\rho(\eta+\alpha)-1} \left(\frac{{(\frac{x}{t})}^{\rho}-1}{\rho}\right)^{\alpha-1} f(t) dt\right|^p dx \right)^\frac{1}{p}\nonumber \\
&=\frac{\rho^{\alpha-\beta}}{\Gamma(\alpha)}\left(\int_a^{b}\left|\int_1^\frac{x}{a} x^{c+\kappa-\frac{1}{p}}\, \left(\frac{x}{u}\right)^{\rho(\eta+\alpha)-1} \left(\frac{u^{\rho}-1}{\rho}\right)^{\alpha-1} f\left(\frac{x}{u}\right) x \frac{du}{u^2}\right|^p dx \right)^\frac{1}{p}\nonumber \\
&\leq\frac{\rho^{\alpha-\beta}}{\Gamma(\alpha)} \int_1^{\frac{b}{a}}\left(\frac{u^{\rho}-1}{\rho }\right)^{\alpha -1}\cdot\frac{1}{u^{\rho(\eta+\alpha)+1}}\left(\int_{at}^b x^{cp} \left|f\left(\frac{x}{u}\right)\right|^p\,\frac{dx}{x}\right)^{\frac{1}{p}} du \cdot b^{\rho(\eta+\alpha)+\kappa}\nonumber \\
 &= \frac{\rho^{\alpha-\beta}b^{\rho(\eta+\alpha)+\kappa}}{\Gamma(\alpha)}\int_1^{\frac{b}{a}}\left(\frac{u^{\rho}-1}{\rho}\right)^{\alpha -1}\cdot\frac{u^c}{u^{\rho(\eta+\alpha)+1}}\left(\int_a^{b/u} \big|t^c f(t)\big|^p\,\frac{dt}{t}\right)^{\frac{1}{p}}du \nonumber 
\end{align}
and hence 
\begin{equation*}
\|{}^\rho_a I^\alpha_{t}f \|_{\textit{X}^p_c} \leq \textsl{K}\|f\|_{\textit{X}^p_c}
\end{equation*}
where
\begin{equation}
   \textsl{K} = \frac{\rho^{1-\beta}b^{\rho(\alpha+\eta)+\kappa}}{\Gamma(\alpha)}\int^{\frac{b}{a}}_1 \frac{u^{c-\rho(\alpha+\eta)-1}}{(u^{\rho}-1)^{1-\alpha}}du, 
	\qquad   1\leq p < \infty
\label{eq:con2}   
\end{equation}
thus, Theorem \ref{eq:th1} is proved for $1\leq p < \infty$. For $p=\infty$, by taking into account the essential supremum \cite[Eq. (3.2)]{udita2}, we have 
\begin{align}
\Big|x^c\big({}^\rho I^{\alpha, \beta}_{a+; \eta, \kappa}f\big)(x)\Big| &\leq \frac{\rho^{1-\beta}b^\kappa}{\Gamma({\alpha})} \int^x_a (x^{\rho} - \tau^{\rho})^{\alpha -1} \tau^{\rho\eta+\rho-1}\Big(\frac{x}{\tau}\Big)^c \big|\tau^cf(\tau)\big| d\tau \nonumber \\
&\leq \frac{\rho^{1-\beta}b^{\rho(\alpha+\eta)+\kappa}}{\Gamma(\alpha)}\int^{\frac{b}{a}}_1 \frac{u^{c-\rho(\alpha+\eta)-1}}{(u^{\rho}-1)^{1-\alpha}}du\cdot\|f\|_{\textit{X}^{\infty}_c}  
\end{align}
after the substitution $u=x/\tau$. This agrees with (\ref{eq:con2}) above. This completes the proof of the theorem.  
\end{proof}	
For simplicity, we have only considered the cases of $0 \leq a < \leq b < \infty$. For $a=-\infty$ and $b=\infty$, we obtain a different $K$ in (\ref{eq:con1}).


Next we give the semigroup properties of the integral operator.

\begin{theorem} \label{eq:th2}
Let $\alpha_1, \alpha_2 >0,\, \beta_1, \beta_2 \in \mathbb{R}, \, 1 \leq p \leq \infty, \, 0 < a < b < \infty $ and let $\rho \in \mathbb{R}$ and $c \in \mathbb{R}$ be such that $\rho \geq c$. Then for $f \in \textit{X}^p_c(a,b)$ the semigroup properties hold. That is,
\begin{equation}\label{eq:semi}
  {}^\rho I^{\alpha_1, \beta_1}_{a+; \eta_1, \kappa_1}{}^\rho I^{\alpha_2, \beta_2}_{a+; \eta_2, -\rho\eta_1}f= {}^\rho I^{\alpha_1+\alpha_2, \beta_1+\beta_2}_{a+; \eta_2, \kappa_1}f \quad \mbox{and} \quad {}^\rho I^{\alpha_1, \beta_1}_{b-; \eta_1, -\rho\eta_2}{}^\rho I^{\alpha_2, \beta_2}_{b-; \eta_2, \kappa_2}f= {}^\rho I^{\alpha_1+\alpha_2, \beta_1+\beta_2}_{b-; \eta_1, \kappa_2}f.
\end{equation}
In particular, we have
\begin{equation}
{}^\rho I^{\alpha_1, \beta_1}_{a+; \eta, \kappa}{}^\rho I^{\alpha_2, \beta_2}_{a+; \eta, -\rho\eta}f= {}^\rho I^{\alpha_1+\alpha_2, \beta_1+\beta_2}_{a+; \eta, \kappa}f \quad \mbox{and} \quad {}^\rho I^{\alpha_1, \beta_1}_{b-; \eta, -\rho\eta}{}^\rho I^{\alpha_2, \beta_2}_{b-; \eta, \kappa}f= {}^\rho I^{\alpha_1+\alpha_2, \beta_1+\beta_2}_{b-; \eta, \kappa}f.
\end{equation}  
\end{theorem}
\begin{proof}
For brevity we only prove the first result. The proof of the other identity is similar. Using Fubini's theorem, for ``sufficiently good'' function $f$, and Dirichlet technique \cite[p.64]{podl}, we have
\begin{align}
 {}^\rho I^{\alpha_1, \beta_1}_{a+; \eta_1, \kappa_1}{}^\rho I^{\alpha_2, \beta_2}_{a+; \eta_2, -\rho\eta_1}f(x) 
       &=\frac{\rho^{2-\beta_1-\beta_2}x^{\kappa_1}}{\Gamma(\alpha_1)\Gamma(\alpha_2)} \int^x_a \frac{\tau^{\rho(\eta_1+1)+\kappa_2-1}}{(x^{\rho } - \tau^{\rho })^{1-\alpha_1}}\int_a^{\tau}\frac{t^{\rho(\eta_2+1)-1}}{(\tau^{\rho } - t^{\rho })^{1-\alpha_2}} f(t) dt d\tau \nonumber\\
			 &=\frac{\rho^{2-\beta_1-\beta_2}x^{\kappa_1}}{\Gamma(\alpha_1)\Gamma(\alpha_2)} \int^x_a t^{\rho(\eta_2+1)-1}f(t) \int_t^x \frac{\tau^{\rho(\eta_1+1)+\kappa_2-1}}{(x^{\rho } - \tau^{\rho })^{1-\alpha_1}(\tau^{\rho } - t^{\rho })^{1-\alpha_2}}d\tau dt \label{eq:pf1}
\end{align}
The inner integral is evaluated by the change of variable $u = (\tau^{\rho}-t^{\rho})/(x^{\rho}-t^{\rho})$ and taking $\kappa_2 = -\rho\eta_1$ into account,
\begin{align}
\int_t^x \frac{\tau^{\rho(\eta_1+1)+\kappa_2-1}}{(x^{\rho } - \tau^{\rho })^{1-\alpha_1}(\tau^{\rho } - t^{\rho })^{1-\alpha_2}}d\tau &=\frac{(x^{\rho}-t^{\rho})^{\alpha_1+\alpha_2-1}}{\rho}\int_0^1(1-u)^{\alpha_1-1}u^{\alpha_2 -1}du,\nonumber\\
&=\frac{(x^{\rho}-t^{\rho})^{\alpha_1+\alpha_2-1}}{\rho}\cdot\frac{\Gamma(\alpha_1)\Gamma(\alpha_2)}{\Gamma(\alpha_1+\alpha_2)}\label{eq:pf2}
\end{align}
according to the known formulae for the beta function \cite{key-8}. Substituting (\ref{eq:pf2}) into (\ref{eq:pf1}) we obtain
\begin{align}
{}^\rho I^{\alpha_1, \beta_1}_{a+; \eta_1, \kappa_1}{}^\rho I^{\alpha_2, \beta_2}_{a+; \eta_2, -\rho\eta_1}f(x)  
               &=\frac{\rho^{1-(\beta_1 +\beta_2)}x^{\kappa_1}}{\Gamma({\alpha_1+\alpha_2})}\int^x_a (x^{\rho}- t^{\rho})^{(\alpha_1 +\alpha_2)-1}t^{\rho(\eta_2+1)-1} f(t)dt,\nonumber\\
               &={}^\rho I^{\alpha_1+\alpha_2, \beta_1+\beta_2}_{a+; \eta_2, \kappa_1}f(x), 
\end{align}
and thus, (\ref{eq:semi}) is proved for ``sufficiently good'' functions $f$. 
If $\rho \geq c$ then by Theorem \ref{eq:th1} the operators $ {}^\rho I^{\alpha_1, \beta_1}_{a+; \eta_1, \kappa_1},\;{}^\rho I^{\alpha_2, \beta_2}_{a+; \eta_2, -\rho\eta_1} $ and ${}^\rho I^{\alpha_1+\alpha_2, \beta_1+\beta_2}_{a+; \eta_2, \kappa_1}$ are bounded in $\textit{X}^p_c(a,b)$, hence the relation (\ref{eq:semi}) is true for $f \in \textit{X}^p_c(a,b).$
This completes the proof of the theorem \ref{eq:th2}.
\end{proof}

We have the following corollary.

\begin{corollary}
Let $\alpha >0,\, \beta >0,\, 1 \leq p \leq \infty, \, 0 < a < b < \infty $ and let $\rho \in \mathbb{R}$ be such that $\rho \geq 1/p$. Then for $f \in \textit{L}^p(a,b)$ the semigroup property (\ref{eq:semi}) holds.
\end{corollary}

Now we shall prove the \textit{fractional product-integration formulae} (\ref{ipart}) for the generalized integral. A similar result is referred by some authors as \textit{fractional integration by parts formula}, but in our opinion this is not similar to the integration by parts formula due to the absence of a derivative term and we shall use the former to identify it.
\begin{theorem} \label{eq:th3}
Let $\alpha >0,\, \beta \in \mathbb{R}, \, 1 \leq p \leq \infty, \, 0 \leq a < b \leq \infty $ and let $\rho \in \mathbb{R}$ and $c \in \mathbb{R}$ be such that $\rho \geq c$. Then for $f, g \in \textit{X}^p_c(a,b)$ the fractional product-integration formula hold. That is,
\begin{equation}\label{eq:proi}
			    \int_a^b x^{\rho-1}f(x)\left({}^\rho I^{\alpha, \beta}_{a+; \eta, \kappa}g\right)(x)dx=\int_a^b x^{\rho-1}g(x)\left({}^\rho I^{\alpha, \beta}_{b-; \eta, \kappa}f\right)(x)dx
			\end{equation}
				Similar results are also valid when $a=0$ and $b=\infty$ and in particular for $\rho=2$ and $\rho=1$.
\end{theorem}
\begin{proof}
The proof is straightforward. Using Dirichlet technique, we have
\begin{align*}
\int_a^b x^{\rho-1}f(x)\left({}^\rho I^{\alpha, \beta}_{a+; \eta, \kappa}g\right)(x)dx &= \frac{\rho^{1-\beta}}{\Gamma(\alpha)}\int_a^b x^{\kappa+\rho-1}f(x)\int_a^x \frac{\tau^{\rho(\eta+1)-1} }{(x^{\rho } - \tau^{\rho })^{1-\alpha}}g(\tau)\,d\tau dx \\
&= \frac{\rho^{1-\beta}}{\Gamma(\alpha)}\int_a^b \tau^{\rho(\eta+1)-1}g(\tau)\int_{\tau}^b \frac{x^{\kappa+\rho-1}}{(x^{\rho } - \tau^{\rho })^{1-\alpha}}f(x)\,dx d\tau \\
&= \int_a^b x^{\rho-1}g(x)\left({}^\rho I^{\alpha, \beta}_{b-; \eta, \kappa}f\right)(x)dx,
\end{align*}
 which completes the proof.
\end{proof}
\begin{remark}Instead of Eq.~(\ref{rint}), we can also consider a more general right- fractional integral given by
\begin{equation}
\left({}^\rho I^{\alpha, \beta}_{b-; \eta, \kappa, \omega}f\right)(x) = \frac{\rho^{1-\beta}x^{\omega}}{\Gamma({\alpha})} \int^b_x \frac{\tau^{\kappa+\rho-1} }{(\tau^{\rho } - x^{\rho })^{1-\alpha}}f(\tau)\,d\tau, \quad (0 \leq a < x <b \leq \infty). 
\label{rint2}
\end{equation}
The drawback is that the results in Theorem~\ref{eq:th3} would take a more complicated form. This explains the rationale behind the choice of the parameter(s) $\rho\eta$ of Eq.~(\ref{rint}).
\end{remark}

\begin{remark}
The generalized fractional integral introduced in this paper has a corresponding generalized fractional derivative which unifies the six fractional derivatives, namely, the Riemann-Liouville, Hadamard, Erd\'elyi-Kober, Katugampola, Weyl and Liouville fractional derivatives in to one form and will be discussed in another article. 
\end{remark}




%
%
%
\bibliographystyle{elsarticle-num}
%
%
%

%
\end{document}